\documentclass[12pt]{amsart}

\usepackage[margin = 1in]{geometry}
\usepackage{amsfonts, amsmath,amscd, amssymb,
latexsym, mathrsfs, mathtools, 
slashed, stmaryrd, verbatim, wasysym, bbm, url}
\usepackage[usenames]{xcolor}
\usepackage{enumerate}
\usepackage[shortlabels]{enumitem}
\usepackage[utf8]{inputenc}

\usepackage[backend=bibtex,bibstyle=authoryear,citestyle=alphabetic,firstinits=true,doi=false,isbn=false,url=false,backref]{biblatex}
\DeclareFieldFormat{labelalphawidth}{\mkbibbrackets{#1}}

\defbibenvironment{bibliography}
  {\list
     {\printtext[labelalphawidth]{%
        \printfield{prefixnumber}%
    \printfield{labelalpha}%
        \printfield{extraalpha}}}
     {\setlength{\labelwidth}{\labelalphawidth}%
      \setlength{\leftmargin}{\labelwidth}%
      \setlength{\labelsep}{\biblabelsep}%
      \addtolength{\leftmargin}{\labelsep}%
      \setlength{\itemsep}{\bibitemsep}%
      \setlength{\parsep}{\bibparsep}}%
      }
  {\endlist}
  {\item}

\newtheorem*{acknowledgements}{Acknowledgements}
\newtheorem*{theorem*}{Theorem}
\newtheorem{theorem}{Theorem}
\newtheorem{corollary}{Corollary}

\newtheorem{lemma}[corollary]{Lemma}

\theoremstyle{definition}
\newtheorem{definition}{Definition}
\newtheorem{example}{Example}
\newtheorem{remark}[example]{Remark}

\numberwithin{equation}{section}

\let\oldsqrt\sqrt
\def\sqrt{\mathpalette\DHLhksqrt}
\def\DHLhksqrt#1#2{%
\setbox0=\hbox{$#1\oldsqrt{#2\,}$}\dimen0=\ht0
\advance\dimen0-0.2\ht0
\setbox2=\hbox{\vrule height\ht0 depth -\dimen0}%
{\box0\lower0.4pt\box2}}

\usepackage{verbatim} 
\DeclareFontFamily{U}{mathx}{\hyphenchar\font45}
\DeclareFontShape{U}{mathx}{m}{n}{
      <5> <6> <7> <8> <9> <10>
      <10.95> <12> <14.4> <17.28> <20.74> <24.88>
      mathx10
      }{}
\DeclareSymbolFont{mathx}{U}{mathx}{m}{n}
\DeclareFontSubstitution{U}{mathx}{m}{n}
\DeclareMathAccent{\widecheck}{0}{mathx}{"71}

\renewcommand{\tilde}[1]{\widetilde{#1}}

\newcommand\eps\varepsilon
\renewcommand\epsilon\varepsilon

\newcommand{\abs}[1]{\left\lvert #1 \right\rvert}

\newcommand\inner[1]{\langle #1 \rangle}

\newcommand\loc{\text{loc}}

\newcommand\Mand{\text{ and }}

\newcommand\paperintro%
        {%
         }
\newcommand\paperbody%
        {%
         }

\newcommand\bbG{\mathbb{G}}
\newcommand\bbH{\mathbb{H}}

\newcommand\bbR{\mathbb{R}}

\newcommand\bbZ{\mathbb{Z}}

\newcommand\cA{\mathcal{A}}

\newcommand\cD{\mathcal{D}}
\newcommand\cE{\mathcal{E}}

\newcommand\cO{\mathcal{O}}

\newcommand\mf[1]{\mathfrak{ #1}}

\DeclareMathAlphabet{\mathpzc}{OT1}{pzc}{m}{it}

\newcommand{\sbs}{\subset}

\DeclareMathOperator{\Div}{div}

\usepackage[refpage]{nomencl}

\makenomenclature

\usepackage{tikz, tikz-cd, tkz-euclide}
\usetikzlibrary{calc,positioning,intersections,quotes,decorations.markings}
\makeatletter
\def\@tocline#1#2#3#4#5#6#7{\relax
  \ifnum #1>\c@tocdepth 
  \else
    \par \addpenalty\@secpenalty\addvspace{#2}%
    \begingroup \hyphenpenalty\@M
    \@ifempty{#4}{%
      \@tempdima\csname r@tocindent\number#1\endcsname\relax
    }{%
      \@tempdima#4\relax
    }%
    \parindent\z@ \leftskip#3\relax \advance\leftskip\@tempdima\relax
    \rightskip\@pnumwidth plus4em \parfillskip-\@pnumwidth
    #5\leavevmode\hskip-\@tempdima
      \ifcase #1
       \or\or \hskip 1em \or \hskip 2em \else \hskip 3em \fi%
      #6\nobreak\relax
    \hfill\hbox to\@pnumwidth{\@tocpagenum{#7}}\par
    \nobreak
    \endgroup
  \fi}
\makeatother

\def\annu#1{_{%
  \vbox{\hrule height .2pt 
    \kern 1pt 
    \hbox{$\scriptstyle {#1}\kern 1pt$}%
  }\kern-.05pt 
  \vrule width .2pt 
}}

\allowdisplaybreaks
\usepackage[colorlinks,citecolor=blue,urlcolor=red,bookmarks=false,hypertexnames=true]{hyperref} 

\title[Super-Poincar\'e inequalities for stratified Lie groups]{Some super-Poincar\'e inequalities for gaussian-like measures on stratified Lie groups}
\author{Yaozhong Qiu}
\address{Department of Mathematics, Imperial College London, 180 Queen’s 
Gate, London, SW7 2AZ, United Kingdom}
\email{y.qiu20@imperial.ac.uk}

\begin{document}
\begin{abstract}
We continue the $U$-bound program initiated in \cite{HZ09} and prove super-Poincar\'e inequalities for a class of subelliptic probability measures defined on M\'etivier groups, the main ingredient in the proof being a Hardy-type inequality. In doing so, we recover and extend some previous results from the probabilistic viewpoint.
\end{abstract}

\maketitle

\section{Introduction and main results}
In this paper, we continue the $U$-bound program developed in \cite{HZ09} and prove super-Poincar\'e inequalities for a class of subelliptic probability measures $d\mu = Z^{-1}e^{-N^p}d\xi$ where $N$ is a homogeneous norm on a stratified Lie group $\bbG$, specifically a M\'etivier group, with explicit estimates, the main ingredient in the proof being a Hardy-type inequality which, to our best knowledge, has not yet appeared in the literature. These measures can be regarded as being analogous to euclidean gaussian-like measures when $N = \abs{x}$ is the euclidean norm on $\bbG = (\bbR^m, +)$, and so our result can be regarded as a subelliptic analogue of the super-Poincar\'e inequality for euclidean superexponential measures. 

If $(X, \cA, \mu)$ is a probability space and $L$ is a selfadjoint operator on $L^2(\mu)$ generating a Markov semigroup $(P_t)_{t \geq 0}$, the super-Poincar\'e inequality (also known as a generalised Nash inequality in \cite[\S7.4.1]{BGL14}) introduced in \cite{Wan00}, is the family of inequalities
\begin{equation}\label{spi0}
    \int f^2d\mu \leq \epsilon \cE(f, f) + \beta_2(\epsilon)\left(\int \abs{f}d\mu\right)^2, \quad \epsilon > \epsilon_0 \geq 0
\end{equation}
where $f: X \rightarrow \bbR$ belongs to the domain $\cD(\cE)$ of the Dirichlet form $\cE$ associated to $L$ and $\beta_2: (\epsilon_0, \infty) \rightarrow \bbR$ is a nondecreasing function called the growth function. As is traditional in the theory of functional inequalities, \eqref{spi0} has spectral theoretic content for $L$, relationships and equivalences between other functional inequalities, and ergodicity implications for $(P_t)_{t \geq 0}$. In particular, it was shown in \cite[Theorem~2.1]{Wan00} that \eqref{spi0} is equivalent to the containment $\sigma_{\text{ess}}(-L) \sbs [\epsilon_0^{-1}, \infty)$ where $\sigma_{\text{ess}}(-L)$ denotes the essential spectrum of $-L$, so that if $\epsilon_0 = 0$ then $-L$ has purely discrete spectrum. In \cite[Theorems~3.1,~3.2]{Wan00} an equivalence was established between \eqref{spi0} and the defective $F$-Sobolev inequality
\begin{equation}\label{fsob0}
    \int f^2F(f^2)d\mu \leq c_1\cE(f, f) + c_2, \quad \int f^2d\mu = 1 \Mand c_1, c_2 > 0
\end{equation}
where $F: [0, \infty) \rightarrow \bbR$ is such that $\lim_{r \to \infty} F(r) = \infty$. Two special growth rates are $\beta_2(\epsilon) \sim e^{c/\epsilon}$, which implies the defective (meaning the right hand side contains an additional $L^2$-term) logarithmic Sobolev inequality, see \cite[\S5]{BGL14}, and $\beta_2(\epsilon) \sim e^{c/\epsilon^\sigma}$ for some $\sigma > 1$, which implies the Lata{\l}a-Oleszkiewicz inequality introduced in \cite{latala2000between}, see also \cite[\S7.6.3]{BGL14}. Lastly, it was shown in \cite[Equation~1.6]{Wan00} and \cite[Proposition~11]{barthe2007isoperimetry} that an analogue of the equivalence between the Poincar\'e inequality and exponential decay to equilibrium holds, namely 
\[ \int (P_tf)^2 d\mu \leq e^{-2t/r}\int f^2d\mu + \beta(r)\left(1 - e^{-2t/r}\right)\left(\int \abs{f}d\mu\right)^2. \]
We refer the reader to \cite{Wan00, Wan02, Wan06, wang2008super, cattiaux2009lyapunov, BGL14} for further discussion of the properties and applications of super-Poincar\'e inequalities (and functional inequalities more generally). 

If $X$ in addition is equipped with a metric $d$, then $U$-bounds, first introduced in \cite{HZ09}, are inequalities of the form 
\begin{equation}\label{ubound0}
    \int \eta f^2d\mu \leq A \cE(f, f) + B \int f^2d\mu, \quad A, B > 0
\end{equation}
where $\eta: X \rightarrow \bbR$ is a function, typically exhibiting growth at infinity (as $d(p, x) \rightarrow \infty$ for some basepoint $p$). These bounds were initially introduced in the setting where $\mu$ is a probability on $\bbR^m$ and $\cE(f, f) = \int \abs{\nabla f}^2d\mu$ with $\nabla = (X_1, \cdots, X_\ell)$ a collection of (possibly noncommuting) vector fields. We will sometimes abuse notation and call $\eta$ itself the $U$-bound, if there is otherwise no ambiguity. One motivation for studying $U$-bounds is that there are rather convenient ways to obtain them. For instance, one way is the original method by \cite{HZ09} which consists of clever use of Leibniz rule, integration by parts, and Young's inequality; this method also provides $U$-bounds in $L^q$, for $q \in (1, 2)$, where $f^2$ and $\abs{\nabla f}^2$ are replaced by $\abs{f}^q$ and $\abs{\nabla f}^q$ respectively. A second (in $L^2$) is the simple method of ``expanding the square" as mentioned in \cite[\S2.3]{RZ22}, while a third is through the semigroup subcommutation for multiplication operators developed in \cite{RZ21, Qiu22}.

The first application of $U$-bounds was to a proof of the usual Poincar\'e inequality \cite[Theorem~3.1]{HZ09} 
\[ \int f^2d\mu - \left(\int fd\mu\right)^2 \leq C \cE(f, f), \]
provided there exists a family of compact sets $(\Omega_M)_{M > 0}$ on which we are given a local Poincar\'e inequality and $\inf_{\Omega_M^c} \eta \rightarrow \infty$ as $M \rightarrow \infty$. In practice, this technical condition is satisfied on $\bbR^m$ when $\mu$ is absolutely continuous with respect to Lebesgue measure and $\cE(f, f) = \int \abs{\nabla f}^2d\mu$ with $\nabla$ the euclidean gradient for $\Omega_M$ a ball of radius $M$ and $\eta$ a radial function diverging to infinity. In the same paper, it was shown, given some additional assumptions on $\eta$, that $U$-bounds pass to logarithmic Sobolev inequalities, more generally $F$-Sobolev type inequalities, and weighted versions thereof. Subsequent applications of $U$-bounds, at least in the subelliptic setting of stratified Lie groups, include, for instance, \cite{inglis2011u} for more entropic and isoperimetric type inequalities, and \cite[Theorem~4.19]{inglis2012spectral} for another Poincar\'e inequality via a $U$-bound $\eta$ which does \emph{not} diverge to infinity in all directions. 
The latter is especially important, in that it can be considered as the starting point of the present work. 

Super-Poincar\'e inequalities for some of the measures of interest in this paper also appeared in \cite[Theorem~4.10]{inglis2012spectral}, although the proof was indirect in the sense it exploited the equivalence between super-Poincar\'e and logarithmic Sobolev inequalities, the latter of which was already proven in \cite[Theorem~4.1]{HZ09}. Lastly, let us note $U$-bounds have also successfully been used to study functional inequalities in the infinite-dimensional setting, see, for instance, the series of related works \cite{IP09, papageorgiou2009logarithmic, inglis2011u, papageorgiou2018log, inglis2019log, konstantopoulos2019log, dagher2022spectral}. We refer the reader also to \cite{dagher2021coercive, CFZ22, BDZ21, CFZ20} for more examples in the finite-dimensional setting. 

Let us now specialise some of the notation introduced to the setting of interest. In this paper, the underlying space $X$ is a (particular type of) stratified Lie group $\bbG$. This is a Lie group $(\bbR^m, \circ)$ with composition law $\circ: \bbR^m \times \bbR^m \rightarrow \bbR^m$ and whose Lie algebra $\mf{g}$ of left invariant vector fields admits the decomposition $\mf{g} = \bigoplus_{i=0}^{r-1} \mf{g}_i$ where the $\mf{g}_i$ are linear subspaces of $\mf{g}$ such that $\mf{g}_i = [\mf{g}_0, \mf{g}_{i-1}]$ for $1 \leq r \leq r-1$. There is a canonical basis $\{X_1, \cdots, X_\ell\}$ for $\mf{g}_0$ whose components form the subgradient $\nabla_\bbG = (X_1, \cdots, X_\ell)$ and the sublaplacian $\Delta_\bbG = \nabla_\bbG \cdot \nabla_\bbG = \sum_{i=1}^\ell X_i^2$, analogous to the euclidean gradient and laplacian respectively. An important consequence of this structure is that the subgradient induces a natural metric on $\bbG$ called the Carnot-Carath\'eodory distance $d = d_\bbG$ analogous to the euclidean norm, while the sublaplacian is associated with a second metric $N$ called the Kor\'anyi-Folland gauge through the distributional identity $\Delta_\bbG N^{2-Q(\bbG)} = \delta_0$ where $Q(\bbG) = \sum_{i=0}^{r-1} (i+1) \dim(\mf{g}_i)$ is the homogeneous dimension of $\bbG$. In other words, $N^{2-Q(\bbG)}$ is the fundamental solution of the sublaplacian. 

When $\bbG = (\bbR^m, +)$ is euclidean with $m \geq 3$, the Kor\'anyi-Folland gauge is the Carnot-Carath\'eodory distance (which is the euclidean norm) up to constants, while on a $H$-type group, which is a particular type of stratified Lie groups generalising the Heisenberg group, the Kor\'anyi-Folland gauge admits an explicit formula and is also called the Kaplan norm. Although they are comparable (meaning they are topologically equivalent), they satisfy different functional inequalities. Indeed, gaussian-like measures defined in terms of the Carnot-Carath\'eodory distance satisfy a logarithmic Sobolev inequality for $p \geq 2$ and a Poincar\'e inequality for $p \geq 1$ by  \cite[Corollaries~3.1~and~4.1]{HZ09} while those defined in terms of the Kor\'anyi-Folland gauge never satisfy a logarithmic Sobolev inequality for any $p \geq 1$ and a Poincar\'e inequality for $p \geq 2$ by \cite[Theorem~6.3]{HZ09} and \cite[Theorem~4.19]{inglis2012spectral} respectively. For more details on stratified Lie groups and homogeneous norms, we refer the reader to \cite[Chapters~1~and~5]{BLU07}.

Let $L = \Delta_\bbG - \nabla_\bbG U \cdot \nabla_\bbG$ for $U \in W^{2, 1}_\loc(\bbG)$ satisfying $Z = \int e^{-U}d\xi < \infty$ with $d\xi$ the Lebesgue measure on $\bbG$. It is well known this implies $L$ is an essentially selfadjoint operator in $L^2(\mu)$ and its extension generates a symmetric Markov semigroup $(P_t)_{t \geq 0}$ in $L^2(\mu)$. It has Dirichlet form $\cE(f, f) = \int \abs{\nabla_\bbG f}^2d\mu$ whose domain is $\cD(\cE) = W^{1, 2}(\mu)$ with $d\mu = \frac{1}{Z}e^{-U}d\xi$, so that the super-Poincar\'e inequality \eqref{spi0} we wish to show is 
\begin{equation}\label{spi1} 
\int f^2d\mu \leq \epsilon\int\abs{\nabla_\bbG f}^2d\mu + \beta_2(\epsilon)\left(\int \abs{f}d\mu\right)^2, \quad \epsilon > \epsilon_0 \geq 0,
\end{equation} 
while the $U$-bounds \eqref{ubound0} we shall use are
\begin{equation}\label{ubound1}
\int \eta f^2d\mu \leq A\int \abs{\nabla_\bbG f}^2d\mu + B\int f^2d\mu, \quad A, B > 0.
\end{equation}

One important point to make is the connection between diffusion and Schr\"odinger operators. For $L = L_0 - \Gamma_0(U, -)$ where $L_0$ is a diffusion operator (meaning a second order differential operator without potential) symmetric in $L^2(e^{-U}\mu_0)$ and $\Gamma_0(f) = \frac{1}{2}(L_0(f^2) - 2fL_0f)$ is its carr\'e du champ, it is known \cite[\S1.15.7]{BGL14} 
\[ \tilde{L} = e^{-U/2}Le^{U/2} = L_0 - V = L_0 - \left(-\frac{L_0U}{2} + \frac{\Gamma_0(U)}{4}\right), \]
i.e. $L$ and $\tilde{L}$ are conjugate to one another. Moreover, $L$ is symmetric in $L^2(e^{-U}\mu_0)$ and multiplication by $e^{U/2}$ is an isometry between $L^2(e^{-U}\mu_0)$ into $L^2(\mu_0)$ so that the operators $L$ and $\tilde{L}$ are unitarily equivalent and share the same spectral content. In particular, $L$ has discrete spectrum if and only if $\tilde{L}$ has discrete spectrum. Moreover, from the viewpoint of Schr\"odinger operators and in light of \cite[Remark~2.3]{RZ21}, this correspondence reveals that a $U$-bound in $L^2$ is essentially a quadratic form lower bound $-\tilde{L} = -L_0 + V \geq V$ viewed from the perspective of diffusion operators, i.e. if $g = e^{U/2}f$ and $V = -\frac{L_0U}{2} + \frac{\Gamma_0(U)}{4}$ then
\[ (-Lg, g)_{L^2(\mu)} = (-e^{-U/2}Le^{U/2}f, f)_{L^2(dx)} = (-\tilde{L}f, f)_{L^2(dx)}\geq (Vf, f)_{L^2(dx)} = (Vg, g)_{L^2(\mu)}. \]

We are motivated by the results of \cite[Theorem~4.19]{inglis2012spectral} who extended the proof of the Poincar\'e inequality beyond the scope of \cite[Theorem~3.1]{HZ09} to allow for degenerate $U$-bounds $\eta$ on $H$-type groups, meaning those $\eta$ not diverging to infinity (and instead vanishing in some directions). Were $\eta$ nondegenerate, for instance when $U = d^p$ with $d$ the Carnot-Carath\'eodory distance on $\bbG$, there is a simple adaptation of the proof of \cite[Theorem~3.1]{HZ09} which recovers the super-Poincar\'e inequality essentially for free. Nonetheless, as explained in the following chapter, it turns out we actually know a super-Poincar\'e inequality exists for some probabilities (including those studied by \cite{inglis2012spectral}) admitting degenerate $U$-bounds, thanks to the equivalence between the super-Poincar\'e inequality and discreteness of spectrum together with a proof of the latter by \cite[Theorem~4.2]{BC17} affirmatively answering a conjecture of \cite{inglis2012spectral} in the setting of the more general M\'etivier groups. Indeed, it was shown measures of the form $d\mu = \frac{1}{Z}e^{-N^p}d\xi$ can be associated with a Schr\"odinger operator whose potential satisfies a discreteness of spectrum criterion first due to \cite[Theorem~3]{Sim08} and generalised later by \cite[Proposition~4.6]{BC22}. We are thus motivated by \cite{BC17} to complete the picture from the other side by providing a (quantitative) proof from the probabilistic viewpoint (via super-Poincar\'e inequalities), and this is the content of Theorem \ref{thm1}. In doing so, we connect the spectral theory of such operators with degenerated quadratic form lower bounds (or potentials from the Schr\"odinger perspective) with Hardy-type inequalities. In the sequel, all integration happens over $\bbG$ unless otherwise stated.

\begin{definition}
    A stratified Lie group $\bbG$ is said to be a M\'etivier group when the Lie algebra $(\mf{g}, [\cdot, \cdot])$ of left-invariant vector fields with center $\mf{z}$, is equipped with an inner product $\inner{\cdot, \cdot}$ such that $[\mf{z}^\perp, \mf{z}^\perp] = \mf{z}$ and for each $Z \in \mf{z}$ the map $J_Z: \mf{z}^\perp \rightarrow \mf{z}$ defined by 
    \[ \inner{J_Z(X), Y} = \inner{Z, [X, Y]}, \quad Y \in \mf{z}^\perp \] 
    is nondegenerate for $Z \neq 0$. Moreover, identifying $\bbG$ with $\bbR^{2n} \times \bbR^m$ via the exponential map for some $n, m \in \bbZ_{\geq 1}$, we may also identify the maps $\{J_Z \mid Z \in \mf{z}\}$ with a family $\{J_t \mid t \in \bbR^m\}$ of $2n$-dimensional skew-symmetric matrices inducing the group law
    \[ (x, z) \cdot (x', z') = \left(x + x', z + z' + \frac{1}{2}\sum_{k=1}^m (J_{e_k}x, x')e_k\right) \]
    for each $(x, z), (x, z') \in \bbR^{2n} \times \bbR^m$ and with $\{e_1, \cdots, e_m\}$ the standard euclidean basis of $\bbR^m$. If in addition $J_Z$ is orthogonal, that is an isometry, then $\bbG$ is said to be a $H$-type group. With this identification, the first strata $\mf{g}_0$ generating the Lie algebra $\mf{g}$ corresponds to $\bbR^{2n}$ and admits a canonical basis $\nabla_\bbG \vcentcolon= (X_1, \cdots, X_{2n})$ comprising the subgradient. We write $\abs{x}$ and $\abs{z}$ for the euclidean norm of $x \in \bbR^{2n}$ and $z \in \bbR^m$ respectively. 
\end{definition}

\begin{definition}
    Let $\bbG \cong \bbR^{2n} \times \bbR^m \ni (x, z)$ be a M\'etivier group. The Kaplan norm $N$ on $\bbG$ is the homogeneous norm 
    \[ N(x, z) = (\abs{x}^4 + 16\abs{z}^2)^{1/4}. \]
    If $\bbG$ is a $H$-type group in the sense of \cite[\S18]{BLU07}, then the Kaplan norm coincides with the Kor\'anyi-Folland gauge due to \cite[Theorem~2]{kaplan1980fundamental}. Note also $\abs{x} \leq N$. 
\end{definition}

\begin{theorem}\label{thm1}
    Let $\bbG$ be a M\'etivier group and consider the probability measure $d\mu = Z^{-1}e^{-N^p}d\xi$ for $p > 2$ with normalisation constant $Z = \int e^{-N^p}d\xi$, Kaplan norm $N$, and Lebesgue measure $d\xi$ on $\bbG$. Then \eqref{spi1} holds with
    \[ \beta_2(\epsilon) \sim \exp(C\epsilon^{-p/(p-2)}), \quad \epsilon_0 = 0 \text{ and some } C > 0.\]
\end{theorem}

In fact, we will be able to prove for $q \in (1, 2)$ the variant 
\begin{equation}\label{spi2}
    \int \abs{f}^qd\mu \leq \epsilon \int \abs{\nabla_\bbG f}^qd\mu + \beta_q(\epsilon)\left(\int \abs{f}^{q/2}d\mu\right)^2, \quad \epsilon > \epsilon_0 \geq 0
\end{equation}
of the super-Poincar\'e inequality, called the $q$-super-Poincar\'e inequality. It is known, thanks to \cite[Lemma~4.10]{inglis2012spectral}, that this $q$-variant is stronger than \eqref{spi1}. Note there are also similar $q$-variants for the Poincar\'e and logarithmic Sobolev inequalities which were again studied using $U$-bounds in \cite{HZ09} and shown to improve (and imply stronger properties for $\mu$) the standard Poincar\'e and logarithmic Sobolev inequalities, see \cite[\S2]{bobkov2005entropy}.

\begin{theorem}\label{thm2}
    Let $\bbG$ be a M\'etivier group and consider the probability measure $d\mu = Z^{-1}e^{-N^p}d\xi$ for $p > 2$ with normalisation constant $Z = \int e^{-N^p}d\xi$, Kaplan norm $N$, and Lebesgue measure $d\xi$ on $\bbG$. Then \eqref{spi2} holds with $q$ H\"older conjugate to $p$ and
    \[ \beta_q(\epsilon) \sim \exp(C\epsilon^{-2(p-1)/(p-2)}), \quad \epsilon_0 = 0 \text{ and some } C > 0.\]
    Moreover, the exponent $2(p-1)/(p-2)$ is optimal in the sense it implies the exponent $p/(p-2)$ in the statement of Theorem \ref{thm1} up to constants.
\end{theorem}

Before we proceed with the proof of these theorems, we note an important intermediate result, namely that the Sobolev inequality on a stratified Lie group (see, for instance, \cite[\S4.7]{varopoulos1993analysis}) together with a linearised interpolation inequality implies the following $q$-super-Poincar\'e inequality for Lebesgue measure.

\begin{lemma}\label{lem:qspilocal}
Let $\bbG$ be a stratified Lie group of homogeneous dimension $Q(\bbG)$ with Lebesgue measure $d\xi$. For $q \in (1, 2]$, it holds
    \begin{equation}\label{eq:qspilocal}
        \int \abs{f}^qd\xi \leq \delta \int \abs{\nabla_\bbG f}^qd\xi + \tilde{\beta}_q(\delta)\left(\int \abs{f}^{q/2}d\xi\right)^2, \quad \tilde{\beta}_q(\delta) \sim 1 + \delta^{-Q(\bbG)/q}. 
    \end{equation}
\end{lemma}

\section{Proof of main results}

\subsection{The nondegenerate case}\label{S2.1}
The proof in the degenerate setting will differ from the methods of \cite{inglis2012spectral} in the sense the latter decomposed the space into three regions. Our proof however will follow the standard decomposition of the space into a ball and its complement. Before we begin let us note that this type of decomposition has been known since at least \cite{Wan00, Wan02, cattiaux2009lyapunov, HZ09}. The methods in these works are in spirit the same but \cite{cattiaux2009lyapunov} uses Lyapunov functions (see also \cite[\S4.6.1]{BGL14}) to control the integral taken over the complement of the ball while the others use quadratic form lower bounds. For completeness, let us consider first the case of the $2$-super-Poincar\'e inequality with the Carnot-Carath\'eodory distance on the Heisenberg group. The argument begins with the following $U$-bound which is the content of \cite[Theorem~2.3]{HZ09}. In the sequel, we write $A \lesssim B$ whenever $A \leq CB$ for some constant $C > 0$ and $f$ always belongs to $\cD(\cE)$.   
\begin{lemma}\label{lem:hzubound}
    Let $\bbG$ be the $1$-Heisenberg group $\bbH^1 \cong \bbR^3$ and consider the probability measure $d\mu = Z^{-1}e^{-d^p}d\xi$ for $p > 1$ with normalisation constant $Z = \int e^{-d^p}d\xi$, Carnot-Carathe\'odory distance $d$, and Lebesgue measure on $d\xi$ on $\bbG$. Then
    \begin{equation}\label{eq:hzubound}
        \int d^{2(p-1)}f^2 d\mu \lesssim \int \abs{\nabla_\bbG f}^2d\mu + \int f^2d\mu.
    \end{equation}
\end{lemma}

Without loss of generality assume $f \geq 0$ and let $B_R = \{x \in \bbG \mid d(x) \leq R\}$ be the $d$-ball centred at the origin (which is also the identity element in any stratified Lie group), write $B_R^c$ for its complement, and also write for simplicity $U = d^p$, so that $d\mu = Z^{-1}e^{-U}d\xi$. Then by the super-Poincar\'e inequality for Lebesgue measure \eqref{eq:qspilocal} we can estimate the left hand side of \eqref{spi1} over $B_R$ by 
\begin{align*}
    \int_{B_R} f^2d\mu &\lesssim \int_{B_R} (fe^{-U/2})^2d\xi \vphantom{\tilde{\beta}_2(\delta) \left(\int_{B_R} fe^{-U/2}d\xi\right)^2} \\
    &\lesssim \delta \int_{B_R} \abs{\nabla_\bbG fe^{-U/2}}^2d\xi + \tilde{\beta}_2(\delta) \left(\int_{B_R} fe^{-U/2}d\xi\right)^2 \\
    &\lesssim \delta \int_{B_R} \abs{\nabla_\bbG f}^2d\mu + \delta \int_{B_R} \abs{\nabla_\bbG U}^2f^2d\mu + \tilde{\beta}_2(\delta) \left(\int_{B_R} fe^{-U/2}d\xi\right)^2
\end{align*}
where $\tilde{\beta}_2(\delta) \lesssim 1+ \delta^{-Q(\bbH^1)/2}$ and $Q(\bbG) = Q(\bbH^1) = 4$. Since $B_R$ is compact one expects to be able to remedy the defective $L^1$-term (with respect to the incorrect measure $e^{-U/2}d\xi$) and also control $\abs{\nabla_\bbG U}^2$ by continuity. Indeed, $\abs{\nabla_\bbG U}^2 \lesssim \abs{\nabla_\bbG d}^2 d^{2(p-1)} \leq R^{2(p-1)}$, since $\abs{\nabla_\bbG d} = 1$ by \cite[Theorem~3.8]{monti2000some}, and $e^{-U/2} = e^{U/2} e^{-U} \leq e^{R^p/2} e^{-U}$. It follows 
\[ \int_{B_R} f^2d\mu \lesssim \delta \int \abs{\nabla_\bbG f}^2d\mu + \delta R^{2(p-1)} \int f^2d\mu + e^{R^p}\tilde{\beta}_2(\delta)\left(\int fd\mu\right)^2 \]
On the other hand, by \eqref{eq:hzubound} we can estimate the left hand side of \eqref{spi1} over $B_R^c$ by
\[ \int_{B_R^c} f^2d\mu \leq \frac{1}{R^{2(p-1)}} \int_{B_R^c} d^{2(p-1)}f^2d\mu \lesssim \frac{1}{R^{2(p-1)}} \left(\int \abs{\nabla_\bbG f}^2d\mu + \int f^2d\mu\right).\]
Let $\gamma = 2(p-1)$. Together this is 
\[ \int f^2d\mu \lesssim \left(\delta + \frac{1}{R^\gamma}\right) \int \abs{\nabla_\bbG f}^2d\mu + e^{R^p}\tilde{\beta}_2(\delta)\left(\int fd\mu\right)^2 + \left(\delta R^\gamma + \frac{1}{R^\gamma}\right)\int f^2d\mu. \]
It suffices to choose $R$ large so that $1/R^\gamma$ is comparable to $\epsilon > 0$, and $\delta$ small, comparable to $\epsilon$, and such that the third addend with coefficient $\cO(\delta R^\gamma)$ can be carried over to the left hand side. Let $R = \epsilon^{-1/\gamma}$ and $\delta = \epsilon/2$ so that $\delta R^\gamma + 1/R^\gamma < 2/3$ for $\epsilon$ sufficiently small and thus
\begin{align*}
    \int f^2d\mu &\lesssim \epsilon \int \abs{\nabla_\bbG f}^2d\mu + e^{R^p}\tilde{\beta}_2(\delta)\left(\int fd\mu\right)^2 \\
    &= \epsilon \int \abs{\nabla_\bbG f}^2d\mu + e^{\epsilon^{-p/\gamma}}(1 + \epsilon^{-Q(\bbG)/2})\left(\int fd\mu\right)^2,
\end{align*}
so that by replacing $\epsilon$ with $\epsilon/K$ for some $K > 0$ large enough, we recover \eqref{spi1} with a growth rate like $\beta_2(\epsilon) \sim \exp(C\epsilon^{-p/\gamma}) = \exp(C\epsilon^{-p/(2(p-1))})$ for some $C > 0$. 

Note this argument is formal, in the sense we have performed this decomposition using characteristic functions $\mathbbm{1}_{B_R}$ and $\mathbbm{1}_{B_R^c}$. However, following the proofs of \cite{Wan00, Wan02, cattiaux2009lyapunov, HZ09}, this argument can be made rigorous by multiplying $f$ by a suitable cutoff (belonging to the domain of $\cE$) such as $\min(1, \max(R - d, 0))$, with the cost being constants. Moreover, it may be somewhat surprising that this elementary approach (of controlling the $L^1$-defect by multiplying and dividing by half the density) is actually sharp in the sense $p/(2(p-1)) \leq 1$ for all $p \geq 2$, implying the expected logarithmic Sobolev inequality for the measure first proven in \cite[Theorem~4.1]{HZ09}, the expected ultracontractivity for $p > 2$ according to \cite[Theorem~5.1]{Wan00}, and the correct concentration of measure property (expressed in terms of exponential integrability) according to \cite[Corollary~6.3]{Wan00}. In addition, it is known by \cite[Theorem~4.10]{inglis2012spectral} for $p \geq 2$ and $q$ H\"older conjugate to $p$ that the stronger $q$-super-Poincar\'e inequality holds and has a growth rate $\beta_q$ which implies this one up to constants. 

\begin{remark}
    This argument also works on general stratified Lie groups; the main ingredient is the fact the Carnot-Carath\'eodory distance satisfies the eikonal equation $\abs{\nabla_\bbG d} = 1$ by \cite[Theorem~3.1]{monti2001surface}.
\end{remark}

\subsection{M\'etivier groups: the $L^2$-setting}

As mentioned in the introduction, the main difference in the degenerate setting is that the following $U$-bound $\eta$ vanishes in certain directions and therefore need not be bounded below on $B_R^c$.

\begin{lemma}\label{lem:inglisubound}
    Let $\bbG \cong \bbR^{2n} \times \bbR^m \ni (x, z)$ be a M\'etivier group and consider the probability measure $d\mu = Z^{-1}e^{-N^p}d\xi$ for $p > 2$ with normalisation constant $Z = \int e^{-N^p}d\xi$, Kaplan norm $N$, and Lebesgue measure $d\xi$ on $\bbG$. Then
    \begin{equation}\label{eq:inglisubound}
        \int \eta f^2 d\mu \vcentcolon= \int \abs{x}^2N^{2(p-2)}f^2 d\mu \lesssim \int \abs{\nabla_\bbG f}^2 d\mu + \int f^2d\mu.
    \end{equation}
\end{lemma}

\begin{proof}
    By association with the Schr\"odinger potential $V = \frac{1}{4}\abs{\nabla_\bbG U}^2 - \frac{1}{2}\Delta_\bbG U$, see \cite[\S1.15.7]{BGL14}, we immediately have $\int V f^2d\mu \leq \int \abs{\nabla_\bbG f}^2d\mu$ where 
    \[ V = \frac{1}{4}p^2N^{2(p-1)}\abs{\nabla_\bbG N}^2 - \frac{1}{2}pN^{p-1}\Delta_\bbG N + \frac{1}{2}p(p-1)N^{p-2}\abs{\nabla_\bbG N}^2. \]
    On a $H$-type group these quantities are explicitly computable where $\abs{\nabla_\bbG N}^2 = \abs{x}^2/N^2$ and $\Delta_\bbG N = (Q(\bbG) - 1) \abs{x}^3/N^3$, so that $V$ is comparable to a polynomial in $N$ whose highest order term is $N^{2(p-1)}$ multiplied by $\abs{\nabla_\bbG N}^2$ and thus we can control it from below up to constants by $\abs{\nabla_\bbG N}^2(N^{2(p-1)} - 1) \gtrsim \eta - 1$. On a M\'etivier group the calculations of \cite[Lemma~4.5]{BC17} show the same is true up to constants. 
\end{proof}

\begin{remark}
    It turns out this potential is one which has ``polynomially thin sublevel sets" in the language of \cite[Theorem~3]{Sim08}, which implies the Schr\"odinger operator $\Delta_\bbG - V$ has discrete spectrum. Actually, some work is required to transfer the original ideas of \cite{Sim08} to the setting of Lie groups, which was done by \cite{BC17}. Thus by exploiting the unitary equivalence between $L = \Delta_\bbG - \nabla_\bbG N^p \cdot \nabla_\bbG$ and $L = \Delta_\bbG - V$, we proceed with the knowledge that a $2$-super-Poincar\'e inequality must exist for the former (whose invariant measure is precisely $\mu$ up to normalisation) by its equivalence with the discrete spectrum of the latter. For more spectral theory of the operators considered by \cite{Sim08}, we refer the reader to \cite{simon1983nonclassical, simon1983some} and also the more recent work \cite{aljahili2023non}.
\end{remark}

We will provide two different proofs of Theorem \ref{thm1}. The first proof exploits H\"older's inequality and a standard Hardy inequality which unfortunately comes with the disadvantage of requiring additional assumptions on the horizontal dimension $2n$ of $\bbG \cong \bbR^{2n} \times \bbR^m$. It will however illustrate the main idea which will inspire the second dimension free proof. Our starting point is a generalisation of a standard Hardy inequality \cite[Theorem~1]{davies1998explicit} to the probability setting whose proof follows \cite[Theorem~2.5]{d2005hardy}. For more Hardy inequalities on stratified Lie groups, we refer the reader to \cite{Amb04a, Amb04, RS17, RS17a, RS19}.

\begin{lemma}\label{lem:hardy(1)}
    Let $\bbG$ be a M\'etivier group and consider the probability measure $d\mu = Z^{-1}e^{-N^p}d\xi$ for $p > 2$ with normalisation constant $Z = \int e^{-N^p}d\xi$, Kaplan norm $N$, and Lebesgue measure $d\xi$ on $\bbG$. If $\Delta_\bbG V \geq 0$ then for $r \in (1, 2]$ it holds  
    \begin{equation}\label{lqhardygeneral}
        \int \abs{\Delta_\bbG V} \abs{f}^r d\mu \lesssim \int \frac{\abs{\nabla_\bbG V}^r}{\abs{\Delta_\bbG V}^{r-1}} \abs{\nabla_\bbG f}^rd\mu + \int \frac{\abs{\nabla_\bbG U \cdot \nabla_\bbG V}^r}{\abs{\Delta_\bbG V}^{r-1}} \abs{f}^rd\mu. 
    \end{equation}
\end{lemma}

\begin{proof}
    We use a standard integration by parts argument. If $L = -\sum_{i=1}^\ell X_i^* X_i = \Div \cdot \nabla$ where $\nabla = (X_1, \cdots, X_\ell)$ are the possibly noncommuting vector fields comprising $L$, then integrating by parts yields 
    \[ \int \abs{f}^r \Div h d\xi = -r \int \abs{f}^{r-2}f \nabla f \cdot h d\xi \]
    for $f$ smooth and compactly supported and $h$ a sufficiently regular vector field. There is a generalisation of this formula when Lebesgue measure $d\xi$ is replaced with a density $d\mu = e^{-U}d\xi$, namely by inserting $e^{-U}$ into the left hand side, integrating by parts, and applying the product rule, we obtain 
    \[ \int \abs{f}^r \Div h(e^{-U}d\xi) = -r\int \abs{f}^{r-2}f \nabla f \cdot h(e^{-U}d\xi) + \int \abs{f}^r \nabla U \cdot h(e^{-U}d\xi). \]
    We will in fact need another generalisation when $d\xi$ is replaced with $\omega d\mu = \omega e^{-U}d\xi$, which we document here for convenience. 
    \begin{equation}\label{q-gen-hardy}
        \int \abs{f}^r \Div h \omega d\mu = -r \int \abs{f}^{r-2}f\nabla f \cdot h\omega d\mu + \int \abs{f}^r \nabla U \cdot h \omega d\mu - \int \abs{f}^r \nabla \omega \cdot h d\mu. 
    \end{equation}
    To recover \eqref{lqhardygeneral}, it suffices to take $\omega \equiv 1$ and, assuming $\Div h$ is nonnegative, writing $A_h = \Div h$ and $s = (r-1)/r$ we find
        \begin{align*}
        \int \abs{f}^rA_h d\mu &\lesssim \int \abs{f}^{r-1}\abs{\nabla f}\abs{h} d\mu + \int \abs{f}^r\nabla U \cdot h d\mu 
        \vphantom{\left(\int \abs{f}^rA_hd\mu\right)^r\left(\left(\int \frac{\abs{h}^r}{A_h^{r-1}}\abs{\nabla f}^rd\mu\right)^{1/r} + \left(\int \abs{f}^r\frac{\abs{\nabla U \cdot h}^r}{A_h^{r-1}}\right)^{1/r}\right) }
        \\ 
        &\leq \int \abs{f}^{r-1}A_h^s\frac{\abs{\nabla f}\abs{h}}{A_h^s} d\mu + \int \abs{f}^{r-1}A_h^s \frac{\abs{f}\abs{\nabla U \cdot h}}{A_h^s}d\mu  
        \vphantom{\left(\int \abs{f}^rA_hd\mu\right)^r\left(\left(\int \frac{\abs{h}^r}{A_h^{r-1}}\abs{\nabla f}^rd\mu\right)^{1/r} + \left(\int \abs{f}^r\frac{\abs{\nabla U \cdot h}^r}{A_h^{r-1}}\right)^{1/r}\right) }
        \\
        &\lesssim \left(\int \abs{f}^rA_hd\mu\right)^r\left(\left(\int \frac{\abs{h}^r}{A_h^{r-1}}\abs{\nabla f}^rd\mu\right)^{1/r} + \left(\int \abs{f}^r\frac{\abs{\nabla U \cdot h}^r}{A_h^{r-1}}\right)^{1/r}\right) 
    \end{align*}
    from which it follows 
    \[ \int A_h \abs{f}^r d\mu \lesssim \int \frac{\abs{h}^r}{A_h^{r-1}}\abs{\nabla f}^rd\mu + \int \frac{\abs{\nabla U \cdot h}^r}{A_h^{r-1}} \abs{f}^rd\mu \]
    by Young's inequality. Specialising to $L = \Delta_\bbG = -\nabla_\bbG^* \cdot \nabla_\bbG$, if $\Delta_\bbG V$ is nonnegative then choosing $h = \nabla_\bbG V$ yields \eqref{lqhardygeneral}. 
\end{proof}

\begin{proof}[Proof of Theorem \ref{thm1} via H\"older's inequality]
    Without loss of generality assume $f \geq 0$ and let $B_R = \{x \in \bbG \mid N(x) \leq R\}$ be the $N$-ball of radius $R$ centred at the origin. We control the integral over $B_R$ in exactly the same way as was done in the nondegenerate setting. On $B_R$, we exploit its compactness by first applying the super-Poincar\'e inequality for Lebesgue measure \eqref{eq:qspilocal}
    \[ \int_{B_R} f^2d\mu \lesssim \delta\int_{B_R} \abs{\nabla_\bbG f}^2d\mu + \delta \int_{B_R} \abs{\nabla_\bbG U}^2f^2d\mu + \tilde{\beta}_2(\delta)\left(\int_{B_R} fe^{-U/2}d\xi\right)^2 \]
    and then multiplying and dividing by $e^{-U/2}$ in the $L^1$-defect, giving 
    \begin{equation}\label{thm1eq1}
        \int_{B_R} f^2d\mu \lesssim \delta\int \abs{\nabla_\bbG f}^2 + \delta R^{2(p-1)} \int f^2d\mu + e^{R^p}\tilde{\beta}_2(\delta)\left(\int fd\mu\right)^2
    \end{equation}
    since $\abs{\nabla_\bbG U}^2 \lesssim \eta \leq R^{2(p-1)}$ on $B_R$. On $B_R^c$, we introduce the following Hardy inequality. Taking $r = 2$ and $V = \log \abs{x}$ in \eqref{lqhardygeneral} and assuming the horizontal dimension $2n$ of $\bbG \cong \bbR^{2n} \times \bbR^m$ is greater than $2$, we have 
    \[ \int \frac{f^2}{\abs{x}^2}d\mu \lesssim \int \abs{\nabla_\bbG f}^2d\mu + \int \abs{x}^6N^{2(p-4)}f^2d\mu \lesssim \int \abs{\nabla_\bbG f}^2d\mu + \int f^2d\mu \]
    since on M\'etivier groups one has $\nabla_\bbG N \cdot \nabla_\bbG \abs{x} \simeq \abs{x}^3/N^3$ following the proof of \cite[Lemma~4.5.6]{Ing10} and $\abs{x}^6N^{2(p-4)} \leq \eta$. By H\"older's inequality we find 
    \begin{align*}
        \int_{B_R^c} f^2d\mu &\leq \left(\int_{B_R^c} \frac{f^2}{\abs{x}^2}d\mu\right)^{1/2}\left(\int_{B_R^c} \abs{x}^2f^2d\mu\right)^{1/2} \vphantom{\left(\int_{B_R^c} \abs{\nabla_\bbG f}^2d\mu + \int_{B_R^c} f^2d\mu\right)^{1/2} \left(\frac{1}{R^{2(p-2)}} \int_{B_R^c} \abs{x}^2N^{2(p-2)}f^2d\mu\right)^{1/2}} \\
        &\lesssim \left(\int_{B_R^c} \abs{\nabla_\bbG f}^2d\mu + \int_{B_R^c} f^2d\mu\right)^{1/2} \left(\frac{1}{R^{2(p-2)}} \int_{B_R^c} \abs{x}^2N^{2(p-2)}f^2d\mu\right)^{1/2} \\
        &\lesssim \frac{1}{R^{p-2}}\left(\int \abs{\nabla_\bbG f}^2d\mu + \int f^2d\mu\right)^{1/2}\left(\int \eta f^2d\mu\right)^{1/2} \vphantom{\left(\int_{B_R^c} \abs{\nabla_\bbG f}^2d\mu + \int_{B_R^c} f^2d\mu\right)^{1/2} \left(\frac{1}{R^{2(p-2)}} \int_{B_R^c} \abs{x}^2N^{2(p-2)}f^2d\mu\right)^{1/2}} \\
        &\lesssim \frac{1}{R^{p-2}} \left(\int \abs{\nabla_\bbG f}^2d\mu + \int f^2d\mu\right) \vphantom{\left(\int_{B_R^c} \abs{\nabla_\bbG f}^2d\mu + \int_{B_R^c} f^2d\mu\right)^{1/2} \left(\frac{1}{R^{2(p-2)}} \int_{B_R^c} \abs{x}^2N^{2(p-2)}f^2d\mu\right)^{1/2}}
    \end{align*} 
    It follows
    \[ \int f^2d\mu \lesssim \left(\delta + \frac{1}{R^{p-2}}\right)\int \abs{\nabla_\bbG f}^2d\mu + e^{R^p}\tilde{\beta}_2(\delta)\left(\int fd\mu\right)^2 + \left(\delta R^{2(p-1)} + \frac{1}{R^{p-2}}\right)\int f^2d\mu \]
    and as before it suffices to choose $R$ large so that $R^{-(p-2)}$ is small and comparable to $\epsilon$, namely $R = \epsilon^{-1/(p-2)}$. We may choose $\delta$ as any positive power of $\epsilon$ such that $\delta R^{2(p-1)} = \cO(\epsilon)$ since, in any case, the asymptotics of $\beta_2$ appear only through the estimate on $R$ since the dependence on $\delta$ is polynomial in $\epsilon$. This yields 
    \[ \int f^2d\mu \lesssim \epsilon \int \abs{\nabla_\bbG f}^2d\mu + e^{\epsilon^{-p/(p-2)}}(1 + \delta^{-Q(\bbG)/2})\left(\int fd\mu\right)^2 \]
    and therefore the expected asymptotic behaviour $\beta_2(\epsilon) \sim \exp(C\epsilon^{-p/(p-2)})$.
\end{proof}

It remains to provide an alternative proof when the horizontal dimension is at most $2$ where $\abs{x}^{-2}$ is not locally integrable and the Hardy inequality does not hold. Roughly speaking, the idea for a dimension free proof is based on a linearised version of the proof via H\"older's inequality. By the scalar inequality $1 \lesssim \epsilon \abs{x}^{-2} + \epsilon^{-1}\abs{x}^2$ we deduce if $2n > 2$ that
\[ \int f^2d\mu \lesssim \int f^2\left(\frac{\epsilon}{\abs{x}^2} + \frac{\abs{x}^2}{\epsilon}\right)d\mu \lesssim \epsilon \int \left(\abs{\nabla_\bbG f}^2d\mu + \int \eta f^2d\mu\right) + \frac{1}{\epsilon} \int \abs{x}^2f^2d\mu . \]
This suggests what is needed is actually a Hardy inequality of the form
\[
    \int f^2d\mu \lesssim \epsilon \int \abs{\nabla_\bbG f}^2d\mu + \frac{1}{\epsilon} \int \abs{x}^2f^2d\mu
\]
whose proof can sidestep the integral $\int \frac{f^2}{\abs{x}^2}d\mu$. 

\begin{proof}[Proof of Theorem \ref{thm1}, dimension free]
    We return to \eqref{q-gen-hardy} with $r = 2$, $\omega = 1$, and $h = x$ so that $\Div h = 2n$. By Young's inequality we find 
    \begin{align*}
        \int f^2d\mu &\lesssim \epsilon \int \abs{\nabla_\bbG f}^2d\mu + \frac{1}{\epsilon} \int \abs{x}^2f^2d\mu + \int \abs{x}^4N^{p-4}f^2d\mu \\
        &\lesssim \epsilon \int \abs{\nabla_\bbG f}^2d\mu + \frac{1}{\epsilon} \int \abs{x}^2f^2d\mu 
    \end{align*}
    since by Young's inequality again $\abs{x}^4 N^{p-4} \leq \abs{x}N^{p-2} \cdot \abs{x} \lesssim \epsilon \eta + \frac{1}{\epsilon} \abs{x}^2$.
    The proof now goes as before on $B_R$, while on $B_R^c$ we obtain 
    \begin{align*}
        \int_{B_R^c} f^2d\mu &\lesssim \epsilon \int_{B_R^c} \abs{\nabla_\bbG f}^2d\mu + \frac{1}{\epsilon} \int \abs{x}^2f^2d\mu \\
        &\lesssim \epsilon \int \abs{\nabla_\bbG f}^2d\mu + \frac{1}{\epsilon}\frac{1}{R^{2(p-2)}} \int_{B_R^c} \abs{x}^2N^{2(p-2)}f^2d\mu \\
        &\lesssim \epsilon \int \abs{\nabla_\bbG f}^2d\mu + \frac{1}{\epsilon}\frac{1}{R^{2(p-2)}} \int \eta f^2d\mu \\
        &\lesssim \left(\epsilon + \frac{1}{\epsilon} \frac{1}{R^{2(p-2)}}\right) \int \abs{\nabla_\bbG f}^2d\mu + \frac{1}{\epsilon} \frac{1}{R^{2(p-2)}} \int f^2d\mu. 
    \end{align*}
    Choosing $R = \epsilon^{-1/(p-2)}$ makes all constants comparable to $\epsilon$ and finishes the proof. 
\end{proof}

\begin{remark}
    Were $d\mu$ replaced with $d\xi$ the constants $\epsilon$ and $\frac{1}{\epsilon}$ can be predicted by a standard scaling argument. But in the probability setting this is not available (since $d\mu$ is not scale invariant). 
\end{remark}

\begin{remark}
    This $2$-super-Poincar\'e inequality does not recover the logarithmic Sobolev inequality since $p/(p-2) > 1$ for $p > 2$, in contrast to the $2$-super-Poincar\'e inequality for Carnot-Carath\'eodory distance. This is not a weakness of our methods since, as mentioned in the introduction, a logarithmic Sobolev inequality is impossible for such measures due to \cite[Theorem~6.3]{HZ09}.
\end{remark}

\subsection{M\'etivier groups: the $L^q$-setting}
One may expect, since measures $Z^{-1}e^{-d^p}d\xi$ defined in terms of Carnot-Carath\'eodory distance satisfy the $q$-logarithmic Sobolev inequality with $q \in (1, 2]$ H\"older conjugate to $p \geq 2$ by \cite[Theorem~4.1]{HZ09}, that the stronger $q$-super-Poincar\'e inequality should analogously hold for measures $d\mu = Z^{-1}e^{-N^p}d\xi$ defined in terms of Kaplan norm. Indeed, this is true, and as expected, boils down to the super-Poincar\'e inequality \eqref{eq:qspilocal} together with a $U$-bound and a Hardy inequality. Although the $U$-bound 
\begin{equation}\label{uboundinglis}
    \int \eta\abs{f}^qd\mu \vcentcolon= \int \abs{x}^2N^{p-2}\abs{f}^qd\mu \lesssim \int \abs{\nabla_\bbG f}^qd\mu + \int \abs{f}^qd\mu
\end{equation}
has already been provided by \cite[Lemma~4.20]{inglis2012spectral}, we will actually prove a stronger $U$-bound (which leads to stronger results). 

It turns out that one can follow the proof with the weaker $U$-bound \eqref{uboundinglis} but it will be ``wrong" in the sense the growth of $\beta_q$ will not be the optimal one claimed by Theorem \ref{thm2}. According to \cite[Lemma~4.11]{inglis2012spectral} one can multiply the exponent of $\epsilon$ in a $q$-super-Poincar\'e inequality by $q/2$ to obtain the exponent in a $2$-super-Poincar\'e inequality. For instance, in the case of the Carnot-Carath\'eodory distance, the correct $q$-super-Poincar\'e inequality has exponent $1$ according to \cite[Theorem~4.10]{inglis2012spectral}, that is $\beta_q$ grows like $\beta_q(\epsilon) \sim \exp(C\epsilon^{-1})$ and multiplication $1 \cdot q/2 = p/(2(p-1))$ recovers the exponent we found in \S\ref{S2.1}. 

\begin{lemma}\label{lem:inglisuboundbetter}
   Let $\bbG$ be a M\'etivier group and consider the probability measure $d\mu = Z^{-1}e^{-N^p}d\xi$ for $p > 2$ with normalisation constant $Z = \int e^{-N^p}d\xi$, Kaplan norm $N$, and Lebesgue measure $d\xi$ on $\bbG$. Then with $q$ H\"older conjugate to $p$ it holds
   \begin{equation}\label{eq:inglisuboundimproved}
        \int \eta \abs{f}^q d\mu \vcentcolon= \int \abs{x}^qN^{p-q}\abs{f}^q d\mu \lesssim \int \abs{\nabla_\bbG f}^q d\mu + \int \abs{f}^qd\mu.
    \end{equation}
\end{lemma}

\begin{proof}
    Without loss of generality assume $f \geq 0$. We follow the proof of \cite[Lemma~4.20]{inglis2012spectral} but instead of taking the inner product of $\nabla_\bbG f^qe^{-U}$ with $\abs{x}^{-1}N\nabla_\bbG N$, we take inner product with $N\abs{\nabla_\bbG N}^2\nabla_\bbG N = \abs{x}^{q-2}N^{3-q}\nabla_\bbG N$ and integrate with respect to Lebesgue measure. This yields 
    \[
        p \int f^q \frac{N^{p-q+2}}{\abs{x}^{2-q}}\abs{\nabla_\bbG N}^2d\mu = \int \frac{N^{3-q}}{\abs{x}^{2-q}}\nabla_\bbG N \cdot \nabla_\bbG f^qd\mu - \int \frac{N^{3-q}}{\abs{x}^{2-q}} \nabla_\bbG N \cdot \nabla_\bbG (f^qe^{-N^p})d\xi. 
    \]
    The left hand side coincides with the left hand side of \eqref{eq:inglisuboundimproved}, that is 
    \[ p \int f^q \frac{N^{p-q+2}}{\abs{x}^{2-q}}\abs{\nabla_\bbG N}^2d\mu = p\int \abs{x}^qN^{p-q}f^qd\mu. \]
    The first addend on the right hand side is controlled by Cauchy-Schwarz and Young's inequality, since
    \begin{align*}
        \int \frac{N^{3-q}}{\abs{x}^{2-q}}\nabla_\bbG N \cdot \nabla_\bbG f^qd\mu &\leq \int \frac{N^{3-q}}{\abs{x}^{2-q}} \abs{\nabla_\bbG N} f^{q-1} \abs{\nabla_\bbG f}d\mu \\
        &= \int \abs{x}^{q-1}N^{2-q} f^{q-1} \abs{\nabla_\bbG f} d\mu \\
        &\lesssim \tau \int \abs{x}^q N^{p-q} f^q d\mu + \frac{1}{\tau} \int \abs{\nabla_\bbG f}^q d\mu 
    \end{align*}
    since $p = q/(q-1)$ and hence $(\abs{x}^{q-1}N^{2-q})^{q/(q-1)} = \abs{x}^qN^{p-q}$. Finally, the second addend on the right hand side can be controlled by integrating by parts, since 
    \[ \int \frac{N^{3-q}}{\abs{x}^{2-q}}\nabla_\bbG N \cdot \nabla_\bbG (f^qe^{-N^p})dx = \int \nabla_\bbG \cdot \left(\frac{N^{3-q}}{\abs{x}^{2-q}}\nabla_\bbG N\right)f^qd\mu \]
    and 
    \begin{align*}
        \nabla_\bbG \cdot \left(\frac{N^{3-q}}{\abs{x}^{2-q}}\nabla_\bbG N\right) &= \frac{N^{3-q}}{\abs{x}^{2-q}} \Delta_\bbG N + \frac{1}{\abs{x}^{2-q}}\nabla_\bbG N^{3-q} \cdot \nabla_\bbG N + N^{3-q} \nabla_\bbG \frac{1}{\abs{x}^{2-q}} \cdot \nabla_\bbG N \\
        &\lesssim \frac{N^{-q}}{\abs{x}^{-q}} + \frac{N^{2-q}}{\abs{x}^{2-q}}\abs{\nabla_\bbG N}^2 + \frac{N^{3-q}}{\abs{x}^{3-q}}\nabla_\bbG \abs{x} \cdot \nabla_\bbG N \\ 
        &\lesssim \frac{N^{-q}}{\abs{x}^{-q}} \vphantom{\nabla_\bbG \cdot \left(\frac{N^{3-q}}{\abs{x}^{2-q}}\nabla_\bbG N\right)} \\
        &\lesssim 1 \vphantom{\nabla_\bbG \cdot \left(\frac{N^{3-q}}{\abs{x}^{2-q}}\nabla_\bbG N\right)}. 
    \end{align*}
    Thus taking $\tau$ sufficiently small and rearranging terms we are done. 
 \end{proof}

As before, we provide two different proofs of Theorem \ref{thm2}. 

\begin{proof}[Proof of Theorem \ref{thm2} via H\"older's inequality]
    The argument follows the same ideas as before. For the estimate over $B_R$, we again apply \eqref{eq:qspilocal} to $fe^{-U/q}$ and since $\abs{\nabla_\bbG U}^q = \abs{x}^qN^{p-q} \leq R^p$ 
    the analogue of \eqref{thm1eq1} is 
    \[ \int_{B_R} f^qd\mu \lesssim \delta \int \abs{\nabla_\bbG f}^qd\mu + e^{R^p}\tilde{\beta}_q(\delta)\left(\int f^{q/2}d\mu\right)^2 + \delta R^p \int f^qd\mu. \]
    For the estimate over $B_R^c$, taking $(1, 2) \ni q$ H\"older conjugate to $p > 2$ and $V = \abs{x}^{2-q}$ in \eqref{lqhardygeneral}, then since the horizontal dimension $2n$ is at least $2 > q$, we have 
    \[ \int \frac{f^q}{\abs{x}^q}d\mu \lesssim \int \abs{\nabla_\bbG f}^qd\mu + \int \abs{x}^{3q}N^{p-3q}f^qd\mu \lesssim \int \abs{\nabla_\bbG f}^qd\mu + \int f^qd\mu \]
    since $\abs{x}^{3q}N^{p-3q} \leq \eta$. By H\"older's inequality we find 
    \begin{align*}
        \int_{B_R^c} f^qd\mu &\leq \left(\int_{B_R^c} \frac{f^q}{\abs{x}^q}d\mu\right)^{1/2} \left(\int_{B_R^c} \abs{x}^qf^qd\mu\right)^{1/2} \\
        &\lesssim \left(\int \abs{\nabla_\bbG f}^qd\mu + \int f^qd\mu\right)^{1/2} \left(\frac{1}{R^{p-q}} \int_{B_R^c} \abs{x}^qN^{p-q}f^qd\mu\right)^{1/2} \\
        &\lesssim \frac{1}{R^{(p-q)/2}} \left(\int \abs{\nabla_\bbG f}^qd\mu + \int f^qd\mu\right). 
    \end{align*}
    Choosing $R = \epsilon^{-2/(p-q)}$ and $\delta$ a suitable positive power of $\epsilon$ and following previous arguments finishes the proof with the expected asymptotic behaviour on $\beta_q$. 
\end{proof}

\begin{proof}[Proof of Theorem \ref{thm2}, dimension free]
    Although the horizontal dimension of a M\'etivier group is at least $2 > q$ and the standard $L^q$-Hardy inequality holds, the scalar inequality $1 \lesssim \epsilon \abs{x}^{-q} + \epsilon^{-1} \abs{x}^q$ suggests the dimension free proof we provided in the $L^2$-setting can be transferred to the $L^q$-setting and would provide a generalisation to horizontal dimension $d \geq 1$. We return to \eqref{q-gen-hardy} with $\omega = \abs{x}^{q-2}$ and $h = x$. By Young's inequality we find for $\alpha > 0$
    \begin{align*}
        (2n + q - 2) \int \frac{\abs{f}^q}{\abs{x}^{2-q}}d\mu &= -q \int \abs{x}^{q-2} \abs{f}^{q-2}f \nabla_\bbG f \cdot x + \int \abs{x}^{q-2} \abs{f}^q \nabla_\bbG U \cdot x d\mu \\
        &\lesssim \alpha^q \int \abs{\nabla_\bbG f}^qd\mu + \frac{1}{\alpha^{q/(q-1)}} \int \abs{x}^q \abs{f}^q d\mu
    \end{align*}
    since $\abs{x}^{q-2}\nabla_\bbG U  \cdot h \leq \abs{\nabla_\bbG U} \cdot \abs{x}^{q-1}$ and $\abs{\nabla_\bbG U}^q \lesssim \eta$. Were the horizontal dimension $2n$ instead $d \geq 1$, then $d + q - 2 > 0$ and $\abs{x}^{q-2}$ is locally integrable since $q - 2 > -1$. Note finally that
    \begin{align*}
        \int \abs{f}^qd\mu &\lesssim \int \abs{f}^q \left(\frac{\delta^{2-q}}{\abs{x}^{2-q}} + \frac{\abs{x}^q}{\delta^q}\right)d\mu \\&\lesssim \int \abs{\nabla_\bbG f}^qd\mu + \int \abs{x}^q\abs{f}^qd\mu \\
        &\lesssim \alpha^q \delta^{2-q} \int \abs{\nabla_\bbG f}^qd\mu + \frac{1}{\alpha^{q/(q-1)}}\delta^{2-q} \int \abs{x}^q\abs{f}^qd\mu + \frac{1}{\delta^q} \int \abs{x}^q\abs{f}^qd\mu 
    \end{align*}
    by the scalar inequality $1 \lesssim \epsilon^{2-q}\abs{x}^{q-2} + \epsilon^{-q}\abs{x}^q$. Choosing $\alpha = \epsilon^{2(q-1)/q^2}$ and $\delta = \epsilon^{1/q}$ yields 
    \[ \int \abs{f}^qd\mu \lesssim \epsilon \int \abs{\nabla_\bbG f}^q + \frac{1}{\epsilon} \int \abs{x}^q\abs{f}^qd\mu. \]
    On $B_R^c$ we obtain 
    \begin{align*}
        \int_{B_R^c} \abs{f}^qd\mu &\lesssim \epsilon \int_{B_R^c} \abs{\nabla_\bbG f}^qd\mu + \frac{1}{\epsilon} \int \abs{x}^q\abs{f}^qd\mu \\
        &\lesssim \epsilon \int \abs{\nabla_\bbG f}^qd\mu + \frac{1}{\epsilon}\frac{1}{R^{p-q}} \int_{B_R^c} \abs{x}^qN^{p-q}\abs{f}^qd\mu \\
        &\lesssim \epsilon \int \abs{\nabla_\bbG f}^qd\mu + \frac{1}{\epsilon}\frac{1}{R^{p-q}} \int \eta \abs{f}^qd\mu \\
        &\lesssim \left(\epsilon + \frac{1}{\epsilon} \frac{1}{R^{p-q}}\right) \int \abs{\nabla_\bbG f}^qd\mu + \frac{1}{\epsilon} \frac{1}{R^{p-q}} \int \abs{f}^qd\mu. 
    \end{align*}
    Choosing $R$ as before finishes the proof. 
\end{proof}

\subsection{The Grushin and Heisenberg-Greiner setting}

In this final chapter, we show how the methods presented in this paper can also be applied to probability measures defined in terms of the Grushin and Heisenberg-Greiner operators together with their respective Kaplan norms. Although these operators cannot in general be realised as the sublaplacian on a stratified Lie group, they have a similar structure and generalise the euclidean and Heisenberg settings respectively. To our best knowledge, $U$-bounds for these operators, although a somewhat straightforward generalisation of the previous computations, have not yet appeared in the literature, whilst Hardy inequalities with respect to Lebesgue measure can be found in \cite{Amb04a} and \cite{Amb05} respectively. 

In the Grushin setting where the subgradient and sublaplacian are respectively (see \cite{Amb04a} for more details)
\[ \nabla_\gamma = (\nabla_x, \abs{x}^\gamma \nabla_y) \Mand \Delta_\gamma = \nabla_\gamma \cdot \nabla_\gamma \] 
for $\gamma \geq 0$ acting on $\bbR^n_x \times \bbR^m_y$ and the Kaplan norm is 
\[ N(x, y) = \left(\abs{x}^{2(1+\gamma)} + (1 + \gamma)^2\abs{y}^2\right)^{1/(2+2\gamma)}, \]
it is readily verified 
\begin{equation}\label{grushinestimates}
    \abs{\nabla_\gamma N} = \frac{\abs{x}^\gamma}{N^\gamma}, \quad \Delta_\gamma N = C(n, m, \gamma) \frac{\abs{x}^{2\gamma}}{N^{2\gamma+1}}, \quad \nabla_\gamma N \cdot \nabla_\gamma \abs{x} = \frac{\abs{x}^{2\gamma+1}}{N^{2\gamma+1}}
\end{equation}
for $C(n, m, \gamma) = n + (1 + \gamma)m$. This constant plays the role of the homogeneous dimension; it may be regarded as a sum of the topological dimension of $\bbR^n_x$ and of $\bbR^m_y$ weighted by $1 + \gamma$. These explicit formulas allow us to prove the $U$-bound 
\[ \int \eta \abs{f}^qd\mu \vcentcolon= \int \abs{x}^{\gamma q}N^{p-q\gamma}\abs{f}^qd\mu \lesssim \int \abs{\nabla_\gamma f}^qd\mu + \int \abs{f}^qd\mu \] 
by taking in the proof of Lemma \ref{lem:inglisuboundbetter} the inner product with $N\abs{\nabla_\gamma N}^{q-2}\nabla_\gamma N$ and developing
\begin{align*}
    p \int N^p \abs{\nabla_\gamma N}^q \abs{f}^qd\mu &= \int N \abs{\nabla_\gamma N}^{q-2} \nabla_\gamma N \cdot \nabla_\gamma \abs{f}^q d\mu \\ &- \int N\abs{\nabla_\gamma N}^{q-2} \nabla_\gamma N \cdot \nabla_\gamma(\abs{f}^q e^{-N^p}) d\xi,
\end{align*} 
as well as the Hardy inequality 
\begin{align*}
    (n + \gamma(q - 1) - 1) \int \frac{\abs{f}^q}{\abs{x}^{1-\gamma(q-1)}}d\mu &\lesssim \alpha^q \int \abs{\nabla_\bbG f}^q + \frac{1}{\alpha^{q/(q-1)}} \int \abs{x}^{\gamma q}\abs{f}^qd\mu 
\end{align*}
which passes through to
\begin{align*}
    \int \abs{f}^qd\mu &\lesssim \alpha^q \delta^{1-\gamma(q-1)} \int \abs{\nabla_\gamma f}^qd\mu + \left(\frac{1}{\alpha^{q/(q-1)}}\delta^{1-\gamma(q-1)} + \frac{1}{\delta^{\gamma q}}\right) \int \abs{x}^{\gamma q}\abs{f}^qd\mu \\
    &\lesssim \epsilon \int \abs{\nabla_\gamma f}^qd\mu + \frac{1}{\epsilon^\gamma} \int \abs{x}^{\gamma q} \abs{f}^qd\mu
\end{align*}
by taking $\omega = \abs{x}^{\gamma(q - 1) - 1}$ and $h = x$ in \eqref{q-gen-hardy} and choosing $\alpha = \epsilon^{(\gamma + 1)(q-1)/q^2}$ and $\delta = \epsilon^{1/q}$. (Note the argument for this inequality requires $1 - \gamma(q - 1) \geq 0$, but once proven for some $\gamma$ it extends to any larger $\gamma$ since $\abs{x} \lesssim \abs{x}^{1+\alpha} + \frac{1}{2}$ for any $\alpha > 0$.) Repeating previous arguments yields $R = \epsilon^{-(\gamma+1)/(p - \gamma q)}$ and therefore we obtain a $q$-super-Poincar\'e inequality \eqref{spi2} with 
\[ \beta_q(\epsilon) \simeq \exp\left(C\epsilon^{-p\sigma}\right) = \exp(C\epsilon^{-(\gamma+1)(p-1)/(p-\gamma-1)}), \]
provided $p > \gamma+1 \geq 2$. If $\gamma < 1$ and $2 > p > \gamma + 1$, in which case $q = \frac{p}{p-1} > 2$, then we obtain instead the $2$-super-Poincar\'e inequality \eqref{spi1} with 
\[ \beta_2(\epsilon) \simeq \exp(Ce^{-p\sigma}) = \exp(C\epsilon^{-p(\gamma+1)/(2(p-\gamma-1))}) \]
via the $U$-bound
\[ \int \eta f^2d\mu \vcentcolon= \int \abs{x}^{2\gamma} N^{2(p-\gamma-1)} f^2d\mu \lesssim \int \abs{\nabla_\gamma f}^2d\mu + \int f^2d\mu \]
obtained by following the proof of Lemma \ref{lem:inglisubound}.

Similarly in the Heisenberg-Greiner setting where the subgradient $\nabla_\zeta$, sublaplacian $\Delta_\zeta =\nabla_\zeta \cdot \nabla_\zeta$, and Kaplan norm $N$ are defined as in \cite[\S2.4~and~\S3.2]{d2005hardy}, it is again readily verified 
\[ \abs{\nabla_\zeta N} = \frac{\abs{x}^{2\zeta-1}}{N^{2\zeta-1}}, \quad \Delta_\zeta N = C(n, \zeta) \frac{\abs{x}^{4\zeta-2}}{N^{4\zeta-1}}, \quad \nabla_\zeta N \cdot \nabla_\zeta \abs{x} = \frac{\abs{x}^{4\zeta-1}}{N^{4\zeta-1}} \] 
for $\zeta \geq 1$ and $C(n, \gamma) = 2n + 2\gamma$. These estimates are identical to those appearing in the Grushin case at $\gamma = 2\zeta-1$ so that by analogy we obtain a $q$-super-Poincar\'e inequality \eqref{spi2} with 
\[ \beta_q(\epsilon) \simeq \exp\left(C\epsilon^{-p\sigma(\zeta)}\right) = \exp(C\epsilon^{-2\zeta(p-1)/(p-2\zeta)}) \] 
provided $p > 2\zeta \geq 2$. 

\begin{remark}
    It is perhaps unsurprising we recover the results that were obtained in the M\'etivier setting at $\gamma = \zeta = 1$ since the only ingredients in the proof the $U$-bound and the Hardy inequality are estimates of the form \eqref{grushinestimates}, all of which at $\gamma = \zeta = 1$ agree modulo constants with those appearing in the M\'etivier setting. 
\end{remark}

\begin{acknowledgements}
\textup{We would like to thank Andreas Malliaris for helpful discussions. We are also grateful to the Hausdorff Research Institute for Mathematics for their support and hospitality where part of the work on this paper was undertaken and some of these discussions took place. The author is supported by the President’s Ph.D. Scholarship of Imperial College London.}
\end{acknowledgements}

\printbibliography

\end{document}